\documentclass[12pt]{amsart}
\usepackage[dvips]{graphicx}
\usepackage{amsmath,graphics}
\usepackage{amsfonts,amssymb}
\usepackage{xypic}
\theoremstyle{plain}

\newtheorem{theorem}{Theorem}[section]
\newtheorem{lemma}[theorem]{Lemma}
\newtheorem{corollary}[theorem]{Corollary}
\newtheorem{proposition}[theorem]{Proposition}

 \theoremstyle{remark}
 \newtheorem{remark}[theorem]{Remark}
\newcommand{\al}{\alpha}
\newcommand{\be}{\beta}

\newcommand{\ga}{\gamma}
\newcommand{\ep}{\epsilon}

\newcommand{\la}{\lambda}

\newcommand{\ZZ}{{\mathbb Z}}

\newcommand{\CC}{{\mathbb C}}
\newcommand{\NN}{{\mathbb N}}

\newcommand{\QQ}{{\mathbb Q}}
\newcommand{\SL}{{\rm SL}}
\newcommand{\GL}{{\rm GL}}
\newcommand{\tr}{\operatorname{\it tr}}
\newcommand{\rk}{\operatorname{rk}}
\newcommand{\Aut}{\operatorname{Aut}}
\newcommand{\Hom}{\operatorname{Hom}}
\newcommand{\Tor}{\operatorname{Tor}}
\newcommand{\sm}{{\setminus}}
\newcommand{\notdiv}{{\not  | }}

\begin{document}
\title{Metabelian SL$(n,\CC)$ representations of knot groups}

\author{Hans U. Boden}
\address{Department of Mathematics, McMaster University,
Hamilton, Ontario} \email{boden@mcmaster.ca}
\thanks{The first named author was supported by a grant from the Natural Sciences and Engineering Research Council of Canada.}

\author{Stefan Friedl}
\address{\begin{tabular}{l} Universit\'e du Qu\'ebec \`a Montr\'eal, Montr\'eal, Qu\'ebec,
and\\ University of Warwick, Coventry, UK\end{tabular}}
\email{sfriedl@gmail.com}


\dedicatory{This paper is dedicated to the memory of Jerry Levine.}

\date{\today}
\begin{abstract}
We give a classification of irreducible metabelian representations from a knot group
into $\SL(n,\CC)$ and $\GL(n,\CC)$.
If the homology of the $n$--fold branched cover of the knot is finite,
we show that every irreducible metabelian  $\SL(n,\CC)$
representation is conjugate to a unitary representation and that the set of conjugacy classes of
such representations is finite. In that case,
we give a formula for this number in terms of the Alexander polynomial  of the knot.
These results are the higher rank generalizations of a result of Nagasato, who recently studied
irreducible, metabelian $\SL(2,\CC)$ representations of knot groups.
Finally we deduce the existence of irreducible metabelian $\SL(n,\CC)$
representations of the knot group for any knot with nontrivial Alexander polynomial.
\end{abstract}
\maketitle

\section{Introduction and statement of results}

Given a knot $K \subset S^3$, we let
$X_K = S^3 \setminus \nu(K)$ denote its complement.
In  \cite{GL89}, Gordon and Luecke show that knots are determined by their complements
by proving
that, if $K_1$ and $K_2$ are knots  and $\varphi: X_{K_1} \to X_{K_2}$ a homeomorphism,
then $K_1$ and $K_2$ are equivalent. 
In the case of prime knots, Whitten 
had proved that the homeomorphism type of
the complement is determined by its fundamental group \cite{Wh87}.
Taken together, these two results reduce the classification of prime knots  to
that of knot groups.

Abstract groups are often better understood through their representations.
For example, the knot invariants coming from finite
representations of the knot group are an effective tool for distinguishing knots and
constructing knot tables.
More delicate information can be obtained from studying their character varieties, see \cite{CS83} and \cite{Kl91}.
For instance, independent work of Dunfield and Garoufalidis \cite{DG04}  and Boyer and Zhang \cite{BZ05}
establish the important result that the $A$-polynomial detects the unknot, and they prove this by  exhibiting
an arc of irreducible $\SL(2,\CC)$ characters on $\pi_1(X_K)$ for any nontrivial knot.

Since abelian representations of the knot group factor through $H_1(X_K)$, they provide
little information beyond that known in classical knot theory. Among the nonabelian representations, the simplest
are those that are trivial on the second commutator subgroup $\pi_1(X_K)^{(2)}$.
Such representations are called {\it metabelian}.
Various aspects of metabelian representations of knot groups have been studied by several authors,
and we refer the reader to the interesting papers \cite{Ha79, AHJ07, Je07}.

In this paper, we study metabelian $\SL(n,\CC)$ and $\GL(n,\CC)$ representations of knot groups. We begin
by recalling a result of Nagasato on irreducible metabelian $\SL(2,\CC)$ representations of
knot groups. Here and throughout this paper, for any knot $K$, we denote its
meridian and longitude by $\mu$ and $\la$, respectively.

\begin{theorem}{\cite[Proposition~1.1~and~Theorem~1.2]{Na07}} \label{thm:na1}  \label{thm:na2}
If $K\subset S^3$ is a knot,  then any irreducible metabelian  representation $\al:\pi_1(X_K)\to \SL(2,\CC)$ satisfies
$$ \tr(\al(\mu))=0\text{ and } \tr(\al(\la))=2.$$
Further, there exist only finitely many conjugacy classes of irreducible metabelian representations
of $\pi_1(X_K)$ into $\SL(2,\CC)$, and their number equals $$ \frac{|\Delta_K(-1)|-1}{2}.$$
\end{theorem}

Note that Xiao--Song Lin (\cite[Proposition~4.2]{Li01})  had  obtained the same count for the number of
conjugacy classes of irreducible metabelian $SU(2)$ representations.
(Lin attributed his result to Fox, and Eric Klassen proved the same formula
when counting binary dihedral representations, cf. \cite[Theorem 10]{Kl91}.)
In particular, Nagasato's result implies that every irreducible metabelian $\SL(2,\CC)$ representation is conjugate to a  unitary one.
This property is what one would expect from representations of finite groups. Although
the  quotient $\pi_1(X_K)/ \pi_1(X_K)^{(2)}$ is not finite, we will show that for many $n$
(for instance, whenever $n$ is prime),
every irreducible metabelian $\SL(n,\CC)$ representation of $\pi_1(X_K)$ factors through a finite group.

Unitary metabelian representations of knot groups have been
classified by the second author in \cite{Fr03, Fr04}, and here we give
the corresponding classification result for
$\SL(n,\CC)$ and $\GL(n,\CC)$.
As a consequence, we derive the following  higher rank analogue of Nagasato's result.
 In order to state it precisely, we introduce some notation.
 Given a knot $K \subset S^3$ and a positive integer $n$, we let $L_n$
 denote the $n$-fold cyclic branched
cover of $S^3$ branched along $K$.

\begin{theorem} \label{mainthm}
Suppose $K\subset S^3$ is a knot and $n$ is a positive integer.
\begin{enumerate}
\item[(i)]
Up to conjugation, any irreducible metabelian representation $\al :\pi_1(X_K)\to \SL(n,\CC)$ satisfies
$\al (\la) = I$ and $$\al(\mu)    = \begin{pmatrix}
  0 &  \dots & 0 &(-1)^{n+1} \\
  1&0 &  \dots &0 \\
  &\ddots &\ddots & \vdots \\
    0& &1 & 0
  \end{pmatrix}.$$ In particular, we have
$$ \tr(\al(\mu))=0 \text{ and } \tr(\al(\la))=n.$$
\item[(ii)] If $H_1(L_n)$ is finite, then the number of conjugacy classes of irreducible
metabelian $\SL(n,\CC)$ representations is finite. In fact the number is given by
\begin{eqnarray*}
&\frac{1}{n} \# \{ \chi:H_1(L_n) \to S^1  \mid
\text{$\chi$ does not factor through}  &\\
& \hspace{2.1in} \text{$H_1(L_n) \to H_1(L_\ell )$ for any  $\ell | n$}  \}.& 
\end{eqnarray*}
\item[(iii)] If $H_1(L_n)$ is finite, then  every
irreducible metabelian $\SL(n,\CC)$ representation of $ \pi_1(X_K)$
is conjugate to a unitary representation.
\end{enumerate}
\end{theorem}

It is a well-known result  (cf. \cite[p.~17]{Go77} or \cite[8.21]{BuZi85}) that the $n$--fold cyclic branched
cover $L_n$ has finite  first homology group
if and only if  no root of the
Alexander polynomial $\Delta_K(t)$ is an $n$-th root of unity.
In that case, the order of the homology group is determined by the Alexander polynomial by the formula
\begin{equation} \label{orderH1}
| H_1(L_n )|=\left| \prod_{j=1}^{n-1} \Delta_K(e^{2\pi ij/n})\right| .
\end{equation}
From $\Delta_K(1)=1$ and some basic algebra, it follows that $\Delta_K(z)\ne 0$
for any prime power root of unity $z$, we therefore deduce the well--known fact that
$H_1(L_n)$ is finite whenever $n$ is a prime power.
The following is now a straightforward corollary to Theorem \ref{mainthm}.
Note that when $n=2$, this recovers Nagasato's formula.

\begin{corollary} \label{cor:mainthm}
If $n=p^k$ is a prime power, then $H_1(L_n)$ is finite for any knot $K$
and the number of conjugacy classes of irreducible metabelian $\SL(n,\CC)$
representations of the knot group is  given by
$$\frac{1}{p^k} \left(\left|H_1(L_{p^k})\right|-\left|H_1(L_{p^{k-1}})\right|\right).$$
If $n$ is prime,  then $L_1 = S^3$ and equation (\ref{orderH1}) implies that this number equals
$$\frac{1}{n}\left(\left|H_1(L_{n})\right|-\left|H_1(L_{1})\right|\right) =
\frac{1}{n}\left(\left| \prod_{j=1}^{n-1}  \Delta_K(e^{2 \pi i j/n})\right|-1 \right).$$
\end{corollary}

If $K$ is a knot with trivial Alexander polynomial, then equation (\ref{orderH1}) shows that $L_n$ is a homology
3-sphere for each $n \geq 1,$ and Theorem \ref{mainthm} implies  there are no
irreducible metabelian $\SL(n,\CC)$ representations of $\pi_1(X_K)$ for any $n \geq 2.$
In fact, it is known (see e.g.~\cite[Theorem 1.2]{Ln02}) that $L_n$ is a homology sphere for all
$n \geq 1$ if and only if $K$ has Alexander polynomial $\Delta_K(t)=1.$
Using this, we get the following conclusion.

\begin{corollary}\label{cor:numrep}
Any knot   with nontrivial Alexander polynomial admits
infinitely many irreducible, pairwise non--conjugate metabelian representations.
\end{corollary}

Corollary \ref{cor:numrep} is a consequence of Theorems \ref{thm:numrepa} and \ref{thm:numrepb}, which give more precise information regarding conjugacy classes  of irreducible metabelian representations for knots
with nontrivial Alexander polynomial.
Our proofs make use of the extensive work on homology groups of $n$-fold branched covers of knots
that began with Gordon's paper \cite{Go72} and was continued in \cite{Ri90,GS91,SW02}, and our existence results
of representations are strengthenings of previous ones (c.f.~ \cite[Corollary 11]{Kl91}).

\bigskip

It is a straightforward exercise to extend all of these results to knots in homology spheres.
In particular, we see that given a knot in a homology sphere with nontrivial Alexander polynomial,
the knot group
admits an irreducible, metabelian $\SL(n,\CC)$ representation for some $n\geq 2.$
Of course, any metabelian representation of the knot group must send the longitude to the identity.
It is interesting to compare this statement to the result \cite[Theorem~1.7]{Fn93} of Frohman,
which shows that if $K$ is a fibered knot of genus $g$
in a rational homology sphere, then there exists an irreducible $SU(n)$ representation
of the knot group sending the longitude to $e^{2 \pi i /n} I$ for some $2\leq n \leq g+1$.


\section{Metabelian representations into $\GL(n,\CC)$} \label{sectionmetab}

\subsection{Preliminaries}

A group $G$ is called metabelian if $G^{(2)}=\{e\}$, where $G^{(n)}$ is the $n$--th term of the
derived series of $G$ which is inductively defined by $G^{(0)}=G$ and $G^{(i+1)}=[G^{(i)},G^{(i)}]$.
We say a representation $\varphi:G\to \GL(n,\CC)$ is metabelian if it factors through $G/G^{(2)}$.

A representation $\rho:G\to \Aut(\CC^n)\cong \GL(n,\CC)$ is called \emph{reducible}
if there exists a proper subspace $V\subset \CC^n$ such that $\rho$ restricts to a representation
$\rho:G\to \Aut(V)$.
Otherwise $\rho$ is called \emph{irreducible} or \emph{simple}.
If $\al$ is the direct sum of simple representations, then $\al$ is called \emph{semisimple}.

Given a representation $\al:G\to \Aut(V)$ we say that $\chi:G\to \CC^*=\CC \sm \{0\}$
is a \emph{weight} if there exists a nontrivial $v\in V$ such that  $\al(g)(v)=\chi(g)(v)$ for all $g\in G$.
For any weight $\chi$, we set
$$ V_\chi=\left\{ v\in V \mid \al(g)(v)=\chi(g)(v) \text{ for all $g\in G$}\right\}.$$
Clearly $V_\chi$ is a nontrivial subspace of $V$. We refer to $V_\chi$ as the \emph{weight space} of $\chi$.
If $\chi_1$ and $\chi_2$ are distinct weights, then it is not hard to show that $V_{\chi_1}\cap V_{\chi_2}=\{0\}$.
Any abelian group has at least one weight, this follows easily from
the fact that every irreducible representation of an abelian group is one--dimensional  (cf. \cite[p.~36]{FS92}).

\subsection{Metabelian quotients of knot groups} \label{sectionclassrep}

Let $K\subset S^3$ be a knot.
In the following we denote by $\widetilde{X_K}$
the infinite cyclic cover of $X_K$ corresponding to the abelianization $\pi_1(X_K)\to \ZZ$.
Therefore $\pi_1(\widetilde{X_K})=\pi_1(X_K)^{(1)}$ and
$$ H_1(X_K;\ZZ[t^{\pm 1}])=H_1(\widetilde{X_K}) \cong \pi_1(X_K)^{(1)}/\pi_1(X_K)^{(2)}.$$
The $\ZZ[t^{\pm 1}]$--module structure is
given on the right hand side
by $t^n\cdot g:=\mu^{-n}g\mu^n$, where $\mu$ is the meridian
of $K$.

Given groups $G$ and $H$ together with a homomorphism $\varphi:G \to \Aut(H)$, the semidirect product
$G \ltimes H$ is
 the group whose underlying set is just $G \times H$ and where the group structure is given by
$$ (g_1,h_1) \cdot (g_2,h_2):=(g_1 \cdot g_2,\varphi(g_2)(h_1) \cdot h_2). $$
For a knot $K$, we set $\pi:=\pi_1(X_K)$ and consider the short exact sequence
$$ 1\to \pi^{(1)}/\pi^{(2)}\to \pi/\pi^{(2)}\to\pi/\pi^{(1)}\to 1. $$
Since $\pi/\pi^{(1)}=H_1(X_K)\cong \ZZ$, this sequence splits and we get isomorphisms
$$ \begin{array}{rcccl} \pi/\pi^{(2)}
&\cong & \pi/\pi^{(1)} \ltimes \pi^{(1)}/\pi^{(2)}
&\cong &\ZZ \ltimes \pi^{(1)}/\pi^{(2)}   \cong \ZZ \ltimes H_1(X_K;\ZZ[t^{\pm 1}]) \\
    g &\mapsto &(g,\mu^{-\ep(g)}g) &\mapsto &(\ep(g),\mu^{-\ep(g)}g), \end{array}  $$
where $n \in \ZZ$ acts by conjugation by $\mu^n$ on $\pi^{(1)}/\pi^{(2)}$ and by multiplication
by $t^n$ on $H_1(X_K; \ZZ[t^{\pm1}])$. Thus,
metabelian representations of $\pi_1(X_K)$ can be viewed as
representations of $\ZZ \ltimes H_1(X_K;\ZZ[t^{\pm 1}])$ and vice--versa.

\subsection{Irreducible $\GL(n, \CC)$ representations of certain semidirect products}
In this section we present a classification of the irreducible $\GL(n,\CC)$ representations
of metabelian groups of the form $\ZZ\ltimes H$, where $H$ is a finitely generated
$\ZZ[t^{\pm 1}]$--module and where $n\in \ZZ$ acts on $H$ by multiplication by $t^n$.

We begin with the prototypical example of a $\GL(n,\CC)$ representation of $\ZZ\ltimes H$.
Fix $n \in \NN$.
Let  $\chi:H  \to \CC^*$ be a character  that factors through $H/(t^n-1)$ and let $z\in \CC^*$.
Then it is easy to verify that
 $$ \begin{array}{rcl} \al=\al_{(z,\chi)} :\ZZ \ltimes H& \to &\GL(n,\CC) \\
    (j,h) &\mapsto &
 \begin{pmatrix}
 0& \dots &0&z \\
 1&0& \dots &0 \\
\vdots &\ddots &  & \vdots\\
    0&\dots &1&0 \end{pmatrix}^j
\begin{pmatrix}
\chi(h) &0& \dots&0 \\
 0&\chi(th) & \dots& 0\\
\vdots&&\ddots & \vdots\\ 0&0&\dots &\chi(t^{n-1}h) \end{pmatrix} \end{array} $$
defines a representation. Note that $\al(n,0)$ is a diagonal matrix with each diagonal entry  equal to $z$.
Also note that $\al$ factors through $\ZZ\ltimes H/(t^n-1)$.

Our first goal is to  determine which representations $\al_{(z,\chi)}$ are irreducible.
We say that a character $\chi:H\to \CC^*$ has \emph{order $k$} if
$$ k=\min\left\{ \ell\in \NN \mid \chi \text{ factors through } H/(t^\ell-1)\right\}.$$
Given a character  $\chi:H\to \CC^*$, let $t^i\chi$ be the character defined by $(t^i\chi)(h)=\chi(t^ih)$.
Obviously, if $\chi$ has order $n$, then $t^n\chi =\chi.$ Conversely, the next lemma shows that the order
of any character $\chi:H\to \CC^*$ which factors through $H/(t^n-1)$ divides $n$.

\begin{lemma} \label{lem:gcd}
If $\chi:H\to \CC^*$ is a character which factors through both $H/(t^n-1)$ and $H/(t^\ell-1)$,
then $\chi$ also factors through $H/(t^{\gcd(n,\ell)}-1)$.
\end{lemma}

\begin{proof}
An easy exercise shows that $\gcd(t^n-1,t^\ell-1)=t^{\gcd(n,\ell)}-1$ in the polynomial ring $\QQ[t]$, and using the Euclidean algorithm, we find polynomials $p(t),q(t)\in \QQ[t]$ with $(t^n-1)p(t)+(t^\ell-1)q(t)=t^{\gcd(n,\ell)}-1$. Since the leading coefficients of
$t^n-1$ and $t^\ell-1$ are units in $\ZZ$, we can arrange that $p(t),q(t)$ lie in $\ZZ[t]$.

For any $h\in H$, we have
\begin{eqnarray*}
\chi\left((t^{\gcd(n,\ell)}-1)h\right)
&=&\chi\big[((t^n-1)p(t)+(t^\ell-1)q(t))h\big]\\
&=&\chi\big[(t^n-1)p(t)h\big]+\chi\big[(t^\ell-1)q(t)h\big]=0.
\end{eqnarray*}
\end{proof}

We can now determine which representations $\al_{(z,\chi)}$ are irreducible.

\begin{lemma} \label{lem:propmetabrep}
Suppose  $\chi:H  \to \CC^*$ is a character that factors through $H/(t^n-1)$ and $z\in \CC^*$.
Then $\al_{(z,\chi)}:\ZZ\ltimes H\to \GL(n,\CC)$ is irreducible if and only if the character $\chi$ has order $n$.
\end{lemma}

\begin{proof}
Throughout this proof, let $\al = \al_{(z,\chi)}$ and  suppose $\chi$ has order $n$.
Denote by $\ga:H\to \GL(n,\CC)$ the restriction of $\al$ to $H=0\times H\subset \ZZ\ltimes H$,
and let $\{e_1,\dots,e_n\}$ be  the standard basis of $\CC^n$.
Then $\ga$ restricts to a representation on $\CC e_i$
that is given by $t^{i-1}\chi$
for $i=1,\ldots, n.$
Since $\chi$ has order $n$ it follows that the characters  $\chi,t\chi,\dots,t^{n-1}\chi:H\to \CC^*$ are pairwise distinct.
Clearly any $H$--invariant subspace of $\CC^n$ must be of the form
$\oplus_{i=1}^r \CC e_{n_i}$ for some $\{n_1,\dots,n_r\}\subset \{1,\dots,n\}$.
Let $T=\al(1,0)$ and notice that $T(e_i)=e _{i+1}$ (with indices taken modulo $n$)
It follows that $\al(\ZZ,0)$ acts transitively on the subspaces $\CC e_i$.
In particular we see that the only proper subrepresentation of $\al$ is the zero space.
This concludes the proof that $\al$ is irreducible.

Now suppose $\chi$ has order $\ell < n$.
It is a straightforward exercise to show that $\al$ is reducible.
We will skip this part of the proof since it is also an immediate consequence of the proof of Theorem \ref{thm:classirrgl}.
\end{proof}

The next result presents a  classification of all irreducible $\GL(n,\CC)$ representations
of semidirect products of the form $\ZZ \ltimes H$.

\begin{theorem} \label{thm:classirrgl}
\begin{enumerate}
\item[(i)] Any irreducible representation $\al:\ZZ \ltimes H \to \GL(n,\CC)$ is  conjugate to
$\al_{(z,\chi)}$ for a character  $\chi:H \to \CC^*$ of order $n$ and some $z\in \CC^*$.
\item[(ii)] If  $\chi_1,\chi_2:H  \to \CC^*$ are characters of order $n$ and  $z_1,z_2 \in \CC^*$,
then $ \al_{(z_1,\chi_1)}$ is conjugate to $ \al_{(z_2,\chi_2)}$ if and only if $z_1=z_2$ and
$\chi_1=t^k\chi_2$ for some $k$.
\end{enumerate}
\end{theorem}

\begin{proof} We first prove (ii).
For convenience, we set $ \al_1=\al_{(z_1,\chi_1)}$ and $ \al_2=\al_{(z_2,\chi_2)}$.
First assume that
$ \al_1$ and $ \al_2$ are conjugate.
Note that
$$ z_1(-1)^{n+1}=\det(\al_1(1,0))=\det(\al_2(1,0))=z_2(-1)^{n+1}.$$
Now denote by $\ga_i$ the restriction of $\al_i$ to  $H$.
Clearly the weights of $\ga_i$ are  given by
$$ \chi_i,t\chi_i,\dots,t^{n-1}\chi_i.$$
Also recall that $t^{n+k}\chi_i=t^k\chi_i$.
Since the sets of weights of $\ga_1$ and $\ga_2$ have to agree it now follows
immediately that $\chi_1=t^k\chi_2$ for some $k$.

Now assume $z_1=z_2$ and $\chi_1=t^k\chi_2$ for some integer $k$.
Then it is easy to check that, for any $(j,h) \in \ZZ\ltimes H,$ we have
$$   \begin{pmatrix}
0& \dots &0&z_1 \\ 1&\dots &0&0 \\ \vdots &\ddots &&\vdots \\  0&\dots &1&0 \end{pmatrix} ^{-k}
\al_1(j,h)
 \begin{pmatrix}
 0& \dots &0&z_1 \\ 1&\dots &0&0 \\ \vdots &\ddots &&\vdots \\  0&\dots &1&0 \end{pmatrix}^k
     = \al_2(j,h).
     $$

We now prove (i).
Let  $\al:\ZZ \ltimes H \to \GL(n,\CC)$ be an irreducible representation.
We denote by $\ga$ the resulting representation $H\to 0\times H\to \ZZ\ltimes H\xrightarrow{\al} \GL(n,\CC)$.

Since $H$ is abelian there exists at least one weight $\chi:H\to \CC^*$.
Let $\ell$ be the order of $\chi$, we write $\ell=\infty$ if $\chi$ does not factor through $H/(t^\ell-1)$ for any $\ell$.
For $i=0,\dots,\ell-1$ we consider the weight spaces
$$ V_i:= \left\{v\in \CC^n \mid \ga(h)(v)=\chi(t^ih)(v).\right\} $$
Since $\chi,t\chi,\dots,t^{\ell-1}\chi$ are different we obtain that
$V_0\oplus V_1\oplus \cdots \oplus V_{\ell-1}$ embeds in $\CC^n$.

Now recall that the group structure of $\ZZ \ltimes H$ is given by
$$ (j_1,h_1)(j_2,h_2)=(j_1+j_2,t^{j_2}h_1+h_2) $$
In particular   for any $h \in H$, we have
$$ (j,0)(0,t^jh)=(j,t^jh)=(0,h)(j,0).$$
Therefore, setting $T=\al(1,0)$, we see that
$$ T^j\al(0,t^jh)=\al(j,0)\al(0,t^jh)
 =\al(j,t^jh)=\al(0,h)\al(j,0)=\al(0,h)T^j.$$
It follows that $T^j \ga(t^jh)=\ga(h)T^j$, and
for $v\in V_{0}$, we see that
$$ \ga(h)T^jv=T^j\ga(t^j h)v=T^j \chi(t^jh)v=\chi(t^jh)T^jv.$$
Hence $T^j$ induces a map $V_0\to V_{j}$ (where we take indices modulo $\ell$)
which is an isomorphism with inverse $T^{-j}$.
This shows $\dim(V_j)=\dim(V_0)\ge 1$ for $j=1,\dots,\ell-1$, which together with
the fact that $V_0\oplus V_1\oplus \cdots \oplus V_{\ell-1}$ embeds in $\CC^n$,
implies that $\ell\leq n$, in particular $\ell$ is finite.

Note that $T^\ell$ induces an isomorphism of $V_0$. Let $v$ be an eigenvector of $T^\ell:V_0 \to V_0$
and let $W$ be the subspace of $\CC^n$
spanned by $\{ v, Tv, \ldots, T^{\ell-1} v\}.$ Clearly $\al$ restricts to a representation of $W$, and
so irreducibility of $\al$ implies   $\ell=\dim(W)=n$. It is straightforward to see that,
in terms of the basis $v,Tv,\dots,T^{n-1}v$, the representation $\al$ is given by $\al_{(z,\chi)}$,
where $z$ is the eigenvalue corresponding to the eigenvector $v$ of $T^\ell$.
\end{proof}

\section{Metabelian representations into $\SL(n,\CC)$}

\subsection{Metabelian $\SL(n,\CC$--representations}

In this section we apply the previous results to give a classification of irreducible
metabelian $\SL(n,\CC)$ representations of knot groups.
We begin with an elementary observation.

\begin{lemma}\label{lem:conj}
If the representations $\al,\be:G\to \SL(n,\CC)$ are conjugate over $\GL(n,\CC)$, then
they are also conjugate over $\SL(n,\CC)$.
\end{lemma}

\begin{proof}
Assume that there exists a matrix $P\in \GL(n,\CC)$ such that $P\al(g)P^{-1}=\be(g)$ for all $g\in G$. Then let
$z$ be an $n$--th root of $\det(P)$. Clearly $\det(z^{-1}P)=1$ and
$(z^{-1}P)\al(g)(z^{-1}P)^{-1}=\be(g)$ for all $g\in G$.
\end{proof}

As before, we suppose $H$ is a finitely generated $\ZZ[t^{\pm 1}]$--module and we consider $\SL(n,\CC)$ representations
of the semidirect product $\ZZ\ltimes H$, where $n\in \ZZ$ acts on $H$ by multiplication by $t^n$.
Throughout this section, we make the additional assumption on $H$
that multiplication by $t-1$ is an isomorphism. Notice first that this holds for the principle application we have in mind.
Indeed, the long exact sequence in homology
$$ \cdots \rightarrow H_{i+1}(X_K;\ZZ)\to H_i(\widetilde{X_K};\ZZ)\xrightarrow{t-1} H_i(\widetilde{X_K};\ZZ)\to H_i(X_K;\ZZ)\rightarrow \cdots $$
 shows that $H=H_1(X_K;\ZZ[t^{\pm 1}])$ has this property.

\begin{lemma} \label{lem:sl}
Let $H$ be a $\ZZ[t^{\pm 1}]$--module such that multiplication by $t-1$ is an isomorphism,
$\chi:H  \to \CC^*$ a character that factors through $H/(t^n-1)$, and $z\in \CC^*$. Set $\al=\al_{(z,\ga)}$.
Then for any $(j,h)\in \ZZ\ltimes H$ we have
$$ \det(\al(j,h))=(-1)^{(n+1)j}z^j.$$
\end{lemma}

\begin{proof}
 Let $\chi$ be a character that factors through $H/(t^n-1)$ and let $(j,h)\in \ZZ\ltimes H$.
 It is straightforward to see that
$$
\det(\al(j,h))=(-1)^{(n+1)j}z^{j} \prod_{i=0}^{n-1} \chi(t^ih)
=(-1)^{(n+1)j}z^{j} \chi\left(\sum_{i=0}^{n-1}t^ih\right).$$
Since multiplication by $t-1$ is an isomorphism on $H$, we have $h' \in H$ with $(t-1)h'=h$.
Thus $$\chi\left(\sum_{i=0}^{n-1}t^ih\right) = \chi\left(\sum_{i=0}^{n-1}t^i(t-1)h'\right) =
\chi\left((t^n-1)h' \right) =1,$$
since $\chi$ factors through $H/(t^n-1)$.
\end{proof}

If $H$ is a $\ZZ[t^{\pm 1}]$--module such that multiplication by $t-1$ is an isomorphism, and
$\chi:H\to \CC^*$ is a character of order $n$, then it follows from
Lemma  \ref{lem:sl} that, for $(j, h) \in \ZZ\ltimes H,$ setting
$$  \al_\chi (j,h) =
 \begin{pmatrix} 0& \dots &0&(-1)^{n+1} \\ 1&\dots &0&0 \\
\vdots &\ddots &&\vdots \\
     0&\dots &1 &0 \end{pmatrix}^j \begin{pmatrix} \chi(h) &0&\dots &0 \\
 0&\chi(th) &\dots &0 \\
\vdots &&\ddots &\vdots \\ 0&0&\dots &\chi(t^{n-1}h) \end{pmatrix}
$$
defines an  $\SL(n,\CC)$ representation.

\begin{theorem} \label{thm:classirrsl}
Let $H$ be a $\ZZ[t^{\pm 1}]$--module such that multiplication by $t-1$ is an isomorphism.
Then the following hold:
\begin{enumerate}
\item[(i)] If $\chi:H\to \CC^*$ is a character of order $n$, then $\al_{\chi}$
defines an irreducible $\SL(n,\CC)$ representation.
\item[(ii)] Given two characters $\chi_1,\chi_2:H\to \CC^*$ of order $n$, the representations
$\al_{\chi_1}$ and $\al_{\chi_2}$ are conjugate if and only if $\chi_1=t^k\chi_2$ for some $k$.
\item[(iii)] For any irreducible representation  $\al:\ZZ\ltimes H\to \SL(n,\CC)$
 there exists a character $\chi:H\to \CC^*$ of order $n$ such that
$\al$ is conjugate to $\al_{\chi}$.
\end{enumerate}
\end{theorem}

The theorem is an immediate consequence of Theorem \ref{thm:classirrgl} and Lemmas \ref{lem:propmetabrep}, \ref{lem:conj}, and \ref{lem:sl}.

\subsection{Proofs of Theorem  \ref{mainthm} and Corollary \ref{cor:mainthm}}

We need two more basic lemmas before we are in a position to prove
Theorem  \ref{mainthm} and Corollary \ref{cor:mainthm}.
The following  is well--known and easy to prove.

\begin{lemma} \label{lem:finite}
Let $A$ be a finite abelian group, then
$$ |\Hom(A,\CC^*)|=|\Hom(A,S^1)|=|\Hom(A,\QQ/\ZZ)|=|A|.$$
\end{lemma}

Let $K$ be a knot and $H=H_1(\widetilde{X_K};\ZZ)=H_1(X_K;\ZZ[t^{\pm 1}])$.
Denote by $W_n$ the $n$--fold cover of $X_K$, which we can view as a subset of $L_n$.
Then the projection from $\widetilde{X_K}$ to $W_n$ induces a map
$$ \pi_n:H_1(\widetilde{X_K};\ZZ)  \to H_1(W_n)\to H_1(L_n).$$
The main properties of this map are stated in the following well--known lemma (cf. e.g. \cite{Fr03}).

\begin{lemma} \label{lem:hlk}
For any $n$ the map $\pi_n$ factors through $H/(t^n-1)$.
Given  $\ell |n$ we have a commutative diagram
$$ \xymatrix{ H/(t^n-1) \ar[d]\ar[r]^-\cong & H_1(L_n)\ar[d] \\
H/(t^\ell-1) \ar[r]^-\cong & H_1(L_\ell),}$$
where the horizontal maps are isomorphisms and the vertical maps are surjections.
\end{lemma}

Using Theorem \ref{thm:classirrsl}  we can now give a proof of Theorem  \ref{mainthm}.

\begin{proof}[Proof of Theorem  \ref{mainthm}]
Let $K\subset S^3$ and $n\in \NN$ and let  $\al:\pi_1(X_K)\to \SL(n,\CC)$ be an irreducible metabelian representation.
 It is well--known that the longitude $\la$ lies in $\pi_1(X_K)^{(2)}$, hence $\al(\la)=I$ and $\tr(\al(\la))=n$. This, together with Theorem \ref{thm:classirrsl} (iii) completes the proof of part (i).

 We now turn to the proof of (ii).
In the following we write  $H=H_1(\widetilde{X_K};\ZZ)$. Recall that multiplication by $t-1$ is an isomorphism on $H$. By
Theorem   \ref{thm:classirrsl}, the number of conjugacy classes of irreducible metabelian $\SL(n,\CC)$ representations
 of $\pi_1(S^3\sm K)$  is given by
$$ N=\# \left\{\chi:H \to \CC^* \mid \chi \text{ of order $n$}\right\}/\sim,$$
where $\chi_1\sim  \chi_2$ if and only if there $\chi_1=t^k\chi_2$ for some $k$.

Any character $\chi:H\to \CC^*$ of order $n$ factors through $H/(t^n-1)$,
which is by Lemma \ref{lem:hlk} isomorphic to $H_1(L_n)$.
Now assume that $H_1(L_n)$ is finite.
It follows immediately from Lemma \ref{lem:finite} that $N$ is finite.
Also note that  the group $\ZZ/n= \langle t \mid t^n=1\rangle$
 acts freely on the set of characters of order $n$,
 hence
$$ N=\frac{1}{n} \# \left\{\chi:H \to S^1 \mid \chi \text{ of order $n$}\right\}.$$
By applying Lemmas \ref{lem:gcd} and \ref{lem:hlk}  we see that
$$ N=\frac{1}{n} \# \left\{\chi:H_1(L_n) \to S^1 \mid \chi \text{ does not factor through $H_1(L_\ell)$ for any $\ell | n$}\right\},$$
as  claimed.

Finally we prove (iii).
Suppose $\al:\ZZ\ltimes H\to \SL(n,\CC)$ is an irreducible representation.
Then by Theorem \ref{thm:classirrsl} the representation $\al$ is conjugate to $\al_{\chi}$ for some character
$\chi:H\to \CC^*$ of order $n$. Note that  $\chi$ factors through $H/(t^n-1)\cong H_1(L_n)$ which is finite by hypothesis.
In particular $\chi(h)$  has finite order for each $h \in H$, i.e. $\chi$ is a unitary character $\chi:H\to S^1\subset \CC^*$.
It is now clear from the definition that $\al_{\chi}$ is a unitary representation.
\end{proof}

\begin{remark}
This proof shows that every irreducible representation $\al:\ZZ\ltimes H \to \SL(n,\CC)$
factors through $\ZZ/n \ltimes H/(t^n-1)$, which is a finite group whenever $H_1(L_n)$ is finite.
\end{remark}

Finally, Corollary \ref{cor:mainthm} is an immediate consequence of the following more general result.
\begin{theorem} \label{thm:count}
Given a knot $K$ and $n \in \NN$ with $H_1(L_n)$ finite,
the number of conjugacy classes of irreducible metabelian $\SL(n,\CC)$ representations of the knot group is  given by
$$\frac{1}{n} \sum_{k | n}  \mu(k)  \left|
H_1\left(L_{n/k)}\right)\right|,$$
where $\mu$ is the M\"{o}bius function.
\end{theorem}

Recall that given $n$ with prime decomposition $n=p_1^{n_1}\cdot \dots \cdot p_s^{n_s}$ with $n_i\geq 1$ and distinct primes $p_1,\dots,p_s$ the M\"obius function is defined as
$$ \mu(n) =\left\{ \begin{array}{ll} 0 &\text{ if $n_i\geq 2$ for some $i$},\\
(-1)^s &\text{ otherwise}.\end{array} \right. $$

\begin{proof}
Throughout the proof, we make repeated use of the general fact that, for any $m | n$, the projection $H_1(L_n) \to H_1(L_m)$ is surjective.

First consider the case $n=p^k$ is a prime power. Then $|H_1(L_{p^k})|$ is automatically finite,  and Theorem \ref{mainthm} implies  that
the number of conjugacy classes of irreducible metabelian representations is finite and is given by
$$ \frac{1}{p^k}\# \left\{\chi:H_1(L_{p^k}) \to S^1 \mid \chi \text{ does not factor through $H_1(L_\ell)$   for any $\ell |p^k$} \right\}.$$
If $\ell | p^k$, then $\ell = p^j$ for some $j<k$. It follows that
the projection $H_1(L_{p^{k-1}}) \to H_1(L_{p^j})$ is surjective, thus any character
$\chi:H_1(L_{p^k}) \to S^1$ which factors through $H_1(L_\ell)$
must also factor through  $H_1(L_{p^{k-1}}).$
Lemma \ref{lem:finite}  and  equation (\ref{orderH1}) now show that
 the number of conjugacy classes of irreducible metabelian representations is  given by
$$ \frac{1}{p^k}\left( |H_1(L_{p^k})| - |H_1(L_{p^{k-1}})|\right)  =\prod_{j=1}^{p^k}\Delta_K(e^{2\pi i j/p^k})-\prod_{j=1}^{p^{k-1}}\Delta_K(e^{2\pi i j/p^{k-1}}).$$
 This agrees with the formula given by the theorem in the case $n=p^k$ is a prime power.
This also proves Corollary \ref{cor:mainthm}.

Next   consider the case $n=p^k q^\ell$, where $p$ and $q$ are distinct primes.
If $H_1(L_ {p^k q^\ell})$ is finite,
then  Theorem \ref{mainthm} implies that
the number of conjugacy classes of irreducible metabelian representations is finite and is given by
$$ \frac{1}{p^k q^\ell}\# \left\{\chi:H_1(L_{p^k q^\ell}) \to S^1 \mid \chi \text{ does not factor through $H_1(L_m)$   for any $m |p^k q^\ell$} \right\}.$$
In this case, any character   $\chi:H_1(L_{p^k q^\ell}) \to S^1$
which factors through $H_1(L_m)$ for some $m | p^k q^\ell$
must also factor through either $H_1(L_{p^k q^{\ell-1}})$ or $H_1(L_{p^{k-1} q^\ell})$.
Further, by Lemma \ref{lem:gcd}, any character $ \chi:H_1(L_{p^k q^\ell}) \to S^1$ which factors through both $H_1(L_{p^k q^{\ell-1}})$ and $H_1(L_{p^{k-1} q^\ell})$ must also factor through
 $H_1(L_{p^{k-1} q^{\ell-1}})$.
Lemma \ref{lem:finite} now shows that
 the number of conjugacy classes of irreducible metabelian representations is  given by
$$ |H_1(L_{p^k q^\ell})|-|H_1(L_{p^k q^{\ell-1}})|-|H_1(L_{p^{k-1} q^\ell})|+|H_1(L_{p^{k-1} q^{\ell-1}})|.$$
This agrees with the formula given by the theorem in the case $n=p^kq^\ell$.

Now consider the general case $n=p_1^{k_1}\cdots p_r^{k_r}$, where $p_1, \dots, p_r$ are distinct primes. We will show that the number of conjugacy classes of irreducible metabelian representations is  given by
$$ \frac{1}{n} \sum_{s=0}^{r}  \sum_{i_1 < \cdots < i_s} (-1)^s \left|
  H_1\left(L_{n/(p_{i_1}...p_{i_s})}\right)\right|.$$
Notice that this formula agrees with the one given in the theorem.

Assume $H_1(L_n)$ is finite, and apply Theorem \ref{mainthm} to see that
the number of conjugacy classes of irreducible metabelian representations is finite and is given by
$$ \frac{1}{n}\# \left\{\chi:H_1(L_n) \to S^1 \mid \chi \text{ does not factor through $H_1(L_m)$   for any $m |n$} \right\}.$$
In this case, since $n=p_1^{k_1}\cdots p_r^{k_r}$, any character   $\chi:H_1(L_{n}) \to S^1$
which factors through $H_1(L_m)$ for some $m | n$
must factor through $H_1(L_{n/p_i})$ for some $1 \leq i \leq r$.
Further,  by Lemma \ref{lem:gcd}, for $1\leq i < j \leq r$, any character $ \chi:H_1(L_{n}) \to S^1$ which factors through $H_1(L_{n/p_i})$ and $H_1(L_{n/p_j})$  must also factor through $H_1(L_{n/(p_i p_j )})$. Repeated application of Lemma \ref{lem:gcd} gives the general statement that,
for $1 \leq i_1 <  \cdots < i_s \leq r$,
any character $ \chi:H_1(L_{n}) \to S^1$ which factors through each of
$H_1(L_{n/p_{i_1}}), H_1(L_{n/p_{i_2}}), \ldots, H_1(L_{n/p_{i_s}})$
must also factor through $H_1(L_{n/(p_{i_1} \cdots p_{i_s})})$.

Using this fact, Lemma \ref{lem:finite}, and the
principle of inclusion-exclusion, we obtain the desired result.
\end{proof}

\subsection{Existence results}

In this section, we prove several results on existence of irreducible metabelian $\SL(n,\CC)$ representations of knot groups  and on existence of faithful metabelian $\SL(n,\CC)$ representations of knot groups. Notice that Theorem \ref{mainthm}
implies that any irreducible metabelian representation $\al:\pi_1(X_K)\to \SL(n,\CC)$
of a knot group sends the meridian to a matrix of order $n$ and as such is not faithful.

We begin with the problem of existence of irreducible
representations.
If $K \subset S^3$ is a knot with trivial Alexander
polynomial, then $\pi_1(X_K)^{(1)} = \pi_1(X_K)^{(2)}$, i.e. any metabelian representation is already abelian.
This shows there are no irreducible
metabelian representations $\al:\pi_1(X_K) \to \SL(n,\CC)$ for any $n \geq 2$
for knots with  trivial Alexander
polynomial.

On the other hand, if the Alexander polynomial is not trivial,
then we will see that there always exist
irreducible metabelian $\SL(n,\CC)$ representations of $\pi_1(X_K)$.
In fact, by using information about the homology groups $H_1(L_n)$
of the $n$-fold branched covers  of $K$, we prove the
existence of infinitely many conjugacy classes of irreducible metabelian representations in many cases.
 Fortunately for us, the
homology groups  $H_1(L_n)$ have been
extensively studied, and we shall make frequent use of the ideas and results from many of the excellent papers on the subject, including \cite{Go72, Ri90, GS91, SW02, Ln02}.

In what follows, we denote by
$X_n$ the character variety of metabelian representations
$\al:\pi_1(X_K) \to \SL(n,\CC),$ and by $X^*_n \subset X_n$ the subvariety
of characters of irreducible representations.
Note that since $X^*_n$ is a variety it either consists of finitely many points
or
it contains positive dimensional components. It follows from Theorem \ref{thm:classirrgl} that  $X^*_n$  consists of finitely many points
if and only if there exist only finitely many conjugacy classes of irreducible
metabelian representations $\pi_1(X_K)\to \SL(n,\CC)$.

The following lemma is a reformulation of Theorem \ref{mainthm}.

\begin{lemma} \label{lem1}
Let $n\in \NN$ such that $H_1(L_n)$ is finite. Then the following hold:
\begin{enumerate}
\item[(i)] the variety $X^*_n$ consists of finitely many points or is empty,
\item[(ii)] any irreducible representation $\pi_1(X_K)\to \SL(n,\CC)$ is conjugate to a unitary representation.
\end{enumerate}
\end{lemma}

 The next lemma treats the case when $H_1(L_n)$ is infinite and shows that
$X^*_n$ is either empty or contains positive dimensional components.

\begin{lemma} \label{lem2}
Let $n$ such that $H_1(L_n)$ is infinite. Then the following are equivalent:
\begin{enumerate}
\item[(a)] there exists an irreducible representation $\pi_1(X_K)\to \SL(n,\CC)$,
\item[(b)] there exist infinitely many  conjugacy classes of irreducible representations $\pi_1(X_K)\to \SL(n,\CC)$, none of which is conjugate to a unitary representation,
\item[(c)] there exist infinitely many  conjugacy classes of irreducible unitary representation $\pi_1(X_K)\to SU(n)$,
\item[(d)] there exists a character $\chi:\Tor (H_1(L_n))\to S^1$ that does not factor through
$\Tor( H_1(L_n)) \to \Tor( H_1(L_\ell))$
for any $\ell |n$  with $b_1(L_\ell)=b_1(L_n)$.
\end{enumerate}
\end{lemma}

\begin{proof}

Given an abelian group $A$ we write
 $R(A)=\Hom(A,\CC^*)$.
Note that $R(A)$ is a complex variety of dimension $\rk(A)$.
If $\ell|n$, the  epimorphism $\pi:H_1(L_n)\to H_1(L_\ell)$ induces an injective map $\pi^*:R(H_1(L_\ell))\to R(H_1(L_n))$.
We say that $\chi\in R(A)$ is unitary if $\chi$ lies
in $U(A):=\Hom(A,S^1)\subset R(A)$, otherwise we say $\chi$ is non--unitary. It follows from Theorem \ref{thm:classirrgl}
that given a character $\chi\in R(H_1(L_n))$ of order $n$, the representation $\al_\chi$
is conjugate to a unitary representation if and only if $\chi$ is  unitary, i.e. if and only if $\chi\in U(H_1(L_n))$.

For any unitary character $\chi: \Tor(H_1(L_n)) \to S^1$, we define $R_\chi(H_1(L_n)) \subset R(H_1(L_n))$ by setting
$$R_\chi(H_1(L_n)) =\{\al :H_1(L_n)\to \CC^* \mid \al |_{\Tor(H_1(L_n))} = \chi \}.$$
Fixing a splitting
$H_1(L_n)=F_n\oplus \Tor(H_1(L_n))$, where $F_n$ is torsion free, we see that any character
$\be:F_n\to \CC^*$ determines a character $\al=\be \oplus \chi:F_n\oplus \Tor(H_1(L_n))\to \CC^*$
with $\al \in R_\chi(H_1(L_n)).$
Thus $R_\chi(H_1(L_n))$ is a complex variety of the same dimension as
$R(F_n),$ indeed  $\dim R_\chi(H_1(L_n))= \rk(H_1(L_n))>0$

We first show (d) implies (b),
so let $\chi : \Tor(H_1(L_n)) \to S^1$ be the character whose existence
is guaranteed by (d).
To conclude (b), it suffices by Theorem \ref{thm:classirrsl}  to show that the set

\begin{equation}\label{Kdense}
R_\chi(H_1(L_n)) \, \,\, \sm \,\,\, \bigcup_{\ell |n} \pi^* R(H_1(L_\ell))
\end{equation}
contains infinitely many non--unitary characters.
If $\ell|n$ with $b_1(L_n)=b_1(L_\ell)$, then
$R_\chi(H_1(L_n)) \cap \pi^* R(H_1(L_\ell))=\varnothing$ by hypothesis.
On the other hand, if $\ell|n$ with $b_1(L_\ell)<b_1(L_n)$, then $\pi^* R(H_1(L_\ell)))$ has codimension at least one in $R(H_1(L_n))$.
It follows that the set (\ref{Kdense}) is open and dense in $R_\chi(H_1(L_n))$ and therefore contains infinitely many non--unitary characters.

To see that (d) implies (c), let $U_\chi(H_1(L_n)) \subset U(H_1(L_n))$ be the corresponding
subset of unitary characters defined by setting
 $$U_\chi(H_1(L_n)) =\{\al :H_1(L_n)\to S^1\mid \al |_{\Tor(H_1(L_n))} = \chi \}.$$

The same argument as above now shows that
$$ U_\chi(H_1(L_n))\, \,  \sm \,\, \bigcup\limits_{\ell |n} \pi^* U(H_1(L_\ell))$$
contains infinitely many unitary characters, and Theorem \ref{thm:classirrsl} applies to show (c).

Now we show that if (d) does not hold, then every character  $\al: H_1(L_n)\to \CC^*$  factors through $H_1(L_n) \to H_1(L_\ell)$
for some $\ell$ with $\ell|n$.
Consequently, there are no characters of order $n$, and
Theorem \ref{thm:classirrsl} shows that  none of (a), (b) or (c) hold.

Suppose $\al: H_1(L_n) \to \CC^*$ and
fix  a splitting $H_1(L_n) = F_n \oplus \Tor(H_1(L_n)),$
where $F_n$ is torsion-free.
Writing $\al=\be\oplus \chi:F_n \oplus \Tor(H_1(L_n)) \to \CC^*$,
by assumption, we see $\chi$ must factor through  $\Tor(H_1(L_n)) \to \Tor(H_1(L_\ell))$ for some  $\ell|n$ with $b_1(L_\ell)=b_1(L_n)$.
Let $\chi':\Tor(H_1(L_\ell))\to S^1$ be the corresponding character.
Since $b_1(L_n)=b_1(L_\ell)$,
the restriction of $\pi_*$  gives an isomorphism from $F_n$ to
some torsion-free submodule of $H_1(L_\ell)$,
and since $\pi_*:H_1(L_n)\to H_1(L_\ell)$ is surjective,
the image $F_\ell = \pi_*(F_n)$
determines a splitting $H_1(L_\ell) = F_\ell \oplus \Tor(H_1(L_\ell))$.
Letting $\be': F_\ell \to \CC^*$ be the character  defined by the condition $\be = \be' \circ \pi_*$,
it follows easily that $\al$ is the pullback of $\al'=\be'\oplus \chi':F_\ell \oplus \Tor(H_1(L_\ell)) \to \CC^*$.
Thus $\al$ factors through $H_1(L_n) \to H_1(L_\ell)$,
and this completes the proof.
\end{proof}

\begin{theorem} \label{thm:numrepa}
Suppose $K \subset S^3$ is a knot with Alexander polynomial such that
a zero of $\Delta_K(t)$ is a root of unity.
Let $m$ be the minimal number such that every root of unity that is a zero of $\Delta_K(t)$ is in fact an $m$--th root of unity.
Let $\la_1(t)$ 
be the first Alexander invariant of $K$.

\begin{enumerate}
\item[(i)]  If $\la_1(t) | (t^m-1)$, then
there exist irreducible metabelian $\SL(n,\CC)$ representations
of the knot group $\pi_1(X_K)$  only for
$2\leq n \leq m$.
Furthermore  the variety $X^*_m$ contains positive dimensional components.

\item[(ii)]  If $\la_1(t) \notdiv (t^m-1)$, then
the variety $X^*_{km}$ contains positive dimensional components  for infinitely many $k \geq 2$.

\item[(iii)] If $(n,m)=1$, then $X_n^*$ is finite or empty. If all zeroes of $\Delta_K(t)$ are roots of unity, then  $X_n^* = \varnothing$ for all $(n,m)=1.$

\item[(iv)] If not all zeroes of $\Delta_K(t)$ are roots of unity, then there exist infinitely many $n$ with $(n,m)=1$ for which $X_n^*$ is nonempty.
\end{enumerate}
\end{theorem}

For the proof we need the following lemma.

\begin{lemma}\label{lem:b1}
Let $m$ be the minimal number such that every root of unity which is a zero of $\Delta_K(t)$ is in fact an $m$--th root of unity.
Then for any $k$ we have $b_1(L_{km})=b_1(L_m)$, and for any $n$ such that $m \notdiv n$ we have $b_1(L_n)<b_1(L_m)$.
\end{lemma}

\begin{proof}[Proof of Lemma \ref{lem:b1}]
We write $\Delta_K(t)=\prod_{i=1}^d (t-z_i) $ for $z_1,\dots,z_d\in \CC$.
It is well--known that
\begin{equation} \label{rankH1}
 b_1(L_n) = \# \{ z_i \mid 1 \leq i \leq d, \; z_i^n=1\}.
\end{equation}
(This formula can  be deduced by studying $H_1(X_K;\CC[t,t^{-1}])=H_1(X_K;\ZZ[t,t^{-1}])\otimes \CC$ and using the fact that $\CC[t,t^{-1}]$ is a PID.)
It is clear from the formula and the definition of $m$ that $b_1(L_n)$ is maximal if and only if $m|n$.
\end{proof}

\begin{proof}[Proof of Theorem \ref{thm:numrepa}]

We first prove (i).
If $\la_1(t)|(t^m-1)$, then the main theorem of \cite{Go72}
shows that the first homology groups of $L_n$  are periodic and moreover
$H_1(L_n) =H_1(L_{(n,m)})$.
This implies that every character $\chi:H_1(L_n) \to \CC^*$ factors
through $H_1(L_n) \to H_1(L_{(m,n)})$. Thus if $n>m,$ no character has order $n$. This
shows that irreducible metabelian
representations $\al:\pi_1(X_K) \to \SL(n,\CC)$ can only occur for  $n$
with  $2 \leq n \leq m$.

It is an immediate consequence of Lemmas \ref{lem2} and \ref{lem:b1} that
 the variety $X^*_m$ contains positive dimensional components

We now prove (ii).
By Lemmas \ref{lem2} and \ref{lem:b1}  we only have to show that
for any $K$  there exists an $k>K$ and a character $\Tor(H_1(L_{km}))\to S^1$ which does not factor through a character $\Tor(H_1(L_{\ell m}))\to S^1$ for any  $\ell|k$.
Set
$M= \sum_{\ell=2}^{K}   |\Tor(H_1(L_{\ell m}))|$ and define
\begin{eqnarray*}
T_{K} &\hspace{-5pt}=\hspace{-5pt}& \{ k \mid \text{ there exists a character ${\chi}:\Tor(H_1(L_{km})) \to \CC^*$ which does not factor } \\
&& \hspace{0.35in} \text{through $\Tor (H_1(L_{km})) \to \Tor(H_1(L_{\ell m}))$ for any $\ell | k$ with $\ell \leq K$} \}.
\end{eqnarray*}
Since $\la_1(t) \notdiv (t^m-1)$, it follows from the combination of  \cite[Theorem~4.7]{Go72}
and \cite[Theorem~2.1]{SW02} that
$|\Tor(H_1(L_{km}))| \to \infty$ as $k \to\infty.$
In particular, there exists $k$ with $k>K$ such that  $|\Tor(H_1(L_{km}))| > M$. Clearly this implies that
$T_{K} \neq \varnothing.$

Let $k = \min(T_K)$ and pick a character $\chi:\Tor(H_1(L_{km})) \to \CC^*$ that does not factor through
$\Tor(H_1(L_{km})) \to \Tor (H_1(L_{\ell m}))$ for any $\ell | k$ with $\ell \leq K$.
But then $\chi$ does  not factor through $\Tor(H_1(L_\ell))$ for any $\ell | k$ with $\ell > K$ by the minimality of $k$. The claim now follows from Lemma \ref{lem2}.

We now turn to the proof of (iii). If $(n,m)=1$, then it follows from (\ref{rankH1}) that
$H_1(L_n)$ is finite. In the case that all zeroes of $\Delta_K(t)$ are roots of unity, then it follows from  \cite[Theorem 4.5]{Go72} that in fact $H_1(L_n)$ is trivial. Claim (iii) now follows immediately from Theorem \ref{mainthm}.

Finally we give a proof of (iv). Let $N\in \NN$. We have to show that there exists $n>N$ with $(n,m)=1$ for which there exists an irreducible
metabelian representation $\al:\pi_1(X_K) \to \SL(n,\CC)$.
Set
$M= \sum_{\ell=2}^{N}   |H_1(L_{\ell})|$ and define
\begin{eqnarray*}
S_{N} &\hspace{-5pt}=\hspace{-5pt}& \{ n\mid \text{ $(n,m)=1$ and there exists  a character ${\chi}:\Tor(H_1(L_n)) \to \CC^*$ which does} \\
&& \hspace{0.45in} \text{not factor through $\Tor (H_1(L_n)) \to \Tor(H_1(L_{\ell}))$ for any $\ell | n$ with $\ell \leq N$} \}.
\end{eqnarray*}
By \cite{Ri90} or \cite{GS91}  (or alternatively by \cite[Theorem~2.1]{SW02})  there exists $n$ with $(n,m)=1$ and $n>N$ such that  $|\Tor(H_1(L_{n}))| > M$. Clearly this implies that
$S_{N} \neq \varnothing.$

Let $n = \min(S_N)$ and pick a character $\chi:\Tor(H_1(L_n)) \to \CC^*$ that does not factor through
$\Tor(H_1(L_n)) \to \Tor(H_1(L_\ell))$ for any $\ell | n$ with $\ell \leq N$.
But then $\chi$ does not factor through $\Tor(H_1(L_\ell))$ for any $\ell | n$ with $\ell > N$ by the minimality of $n$. The claim now follows from Theorem \ref{mainthm}.
\end{proof}

The proof of Theorem \ref{thm:numrepa} (iii) and (iv) (setting $m=1$) also immediately gives the following result.

\begin{theorem} \label{thm:numrepb}
Suppose $K \subset S^3$ is a knot with Alexander polynomial $\Delta_K(t) \neq 1$ and
such that no zero of  $\Delta_K(t)$ is a root of unity.
Then for any $n$ the set  $X_n^*$ is finite or empty. Further, $X^*_n$ is nonempty for infinitely many $n$.
\end{theorem}


Now we turn our attention to the problem of existence of faithful  representations
of the metabelian quotient of a knot group.
First note that  it follows from Theorem \ref{mainthm} that given a semisimple  representation $\al:\pi_1(X_K)/\pi_1(X_K)^{(2)}\to \SL(n,\CC)$ we have $\al(\mu^n)=I$, i.e.
$\al$ can not be faithful.
Also note that a unitary representation is necessarily semisimple, hence by the above can not be faithful.

On the other hand if we study non--semisimple non--unitary representations,
then we can always find one that is faithful. The next result produces a faithful
representation into $\GL(n,\CC)$, but it can easily be modified to give
a faithful reducible representation into $\SL(n+2,\CC)$.

\begin{proposition} \label{prop:faithful}
Given any knot  $K\subset S^3$  there exists a faithful reducible representation of $\pi_1(X_k)/\pi_1(X_K)^{(2)}$.
\end{proposition}

Note that representations similar to the ones used in our proof appear also in \cite{Je07}.

\begin{proof}
It is well--known that $H_1(X_K;\ZZ[t^{\pm 1}])$ is $\ZZ$--torsion free, in particular we get an injection
$H_1(X_K;\ZZ[t^{\pm 1}])\to  H_1(X_K;\ZZ[t^{\pm 1}])\otimes \CC =H_1(X_K;\CC[t^{\pm 1}])$.
Since $\CC[t^{\pm 1}]$ is a PID, we have an isomorphism
$$H_1(X_K;\CC[t^{\pm 1}]) \cong \bigoplus\limits_{i=1}^n  \CC[t^{\pm 1}]/(t-z_i)^{r_i}$$ of
$ \CC[t^{\pm 1}]$--modules for some
$z_i\in \CC, r_i \in \NN$.
Note that this gives rise to an injective group homomorphism
\[ \ZZ \ltimes H_1(X_K;\CC[t^{\pm 1}]) \to  \prod\limits_{i=1}^n \ZZ \ltimes \CC[t^{\pm 1}]/(t-z_i)^{r_i},\]
here $n\in \ZZ$ acts on $\CC[t^{\pm 1}]/(t-z_i)^{r_i}$ by multiplication by $z_i^n=t^n$.
By taking direct sums of representations it therefore suffices to find a faithful representation of
$\ZZ \ltimes \CC[t^{\pm 1}]/(t-z)^{r}$ for some $z\in \CC, r\in \NN$.
Given an element $p\in \CC[t^{\pm 1}]/(t-z)^{r}$ there exist unique $a_0,\dots,a_{r-1}\in \CC$
such that $p$ is represented by $\sum_{i=0}^{r-1} a_i(t-z)^i$.
Let $x\in S^1$ an element of infinite order. We then consider the representation
$ \al: \ZZ\ltimes \CC[t^{\pm 1}]/(t-z)^{r}\to  \GL(r+1,\CC)$ defined by
\[
(0,p)\mapsto
\begin{pmatrix}
1 & a_0&a_1&\cdots&a_{r-1} \\
0 & 1 & 0 &\cdots & 0 \\
\vdots & &1&&\vdots \\
&&&\ddots& 0\\
0&\cdots &&0&1
\end{pmatrix}
\mbox{ and } (1,0)\mapsto
x \cdot \begin{pmatrix}
1 & 0&0&\cdots&0 \\
0 & z & 1 &\cdots & 0 \\
\vdots & &z&\ddots&\vdots \\
&&&\ddots& 1\\
0&\cdots &&0&z
\end{pmatrix}.
\]
The group structure on $\ZZ\ltimes \CC[t^{\pm 1}]/(t-z)^{r}$ is given by $(n',h')+(n,h) = (n'+n, z^{n}h' + h).$
Using this, it is not difficult to check that $\al$ is indeed a representation.
It is easy to verify that  $\al$ is also a  faithful representation.
\end{proof}

\end{document}